\documentclass[a4paper]{amsart}

\usepackage{latexsym, amssymb,amsmath, amsthm,amsfonts,color,tikz}
\newcommand{\kk}{\Bbbk}
\newcommand{\kv}{{\kk[\mathcal{V}]}}

\newcommand{\kvg}{{\kk[\mathcal{V}]^{G}}}

\newcommand{\vv}{\mathcal{V}}
\newcommand{\uu}{\mathcal{U}}

\newcommand{\M}{{\mathcal{M}}}

\newcommand{\bA}{{\mathbf{A}}}
\newcommand{\bB}{{\mathbf{B}}}

\newcommand{\ba}{{\mathbf{a}}}
\newcommand{\bb}{{\mathbf{b}}}
\newcommand{\bc}{{\mathbf{c}}}
\newcommand{\bd}{{\mathbf{d}}}
\newcommand{\br}{{\mathbf{r}}}

\newcommand{\cc}{\mathcal{C}}

\newcommand{\ww}{\mathcal{W}}

\def\SL{\operatorname{SL}}

\def\GL{\operatorname{GL}}
\def\Ga{{\mathbb G}_{a}}

\def\chr{\operatorname{char}}

\def\Z{\mathbb{Z}}

\def\C{\mathbb{C}}

\def\codim{\operatorname{codim}}

\def\Tr{\operatorname{Tr}}

\def\rk{\operatorname{rk}}
\newtheorem{Lemma}{Lemma}[section]
\newtheorem{Theorem}[Lemma]{Theorem}

\newtheorem{cor}[Lemma]{Corollary}
\newtheorem{prop}[Lemma]{Proposition}
\newtheorem{conj}[Lemma]{Conjecture}

\theoremstyle{definition}
  \newtheorem{Def}[Lemma]{Definition}  %needs a capital as def already defined

\theoremstyle{remark}
  \newtheorem{rem}[Lemma]{Remark}
\newtheorem{eg}[Lemma]{Example}

\newtheoremstyle{Acknowledgments}% name
  {}% {\topsep}%      Space above
    {}% {\topsep}%      Space below
     {}%         Body font
     {}%         Indent amount (empty = no indent, \parindent = para indent)
    {\bfseries}% Thm head font
    {}%        Punctuation after thm head
     {.5em}%     Space after thm head: " " = normal interword space;
     {\thmname{#1}\thmnumber{ }\thmnote{ (#3)}}%         Thm head spec (can be left empty, meaning `normal')
\theoremstyle{Acknowledgments}
\newtheorem{ack}{Acknowledgments.}

\title{The separating variety for matrix semi-invariants}

\author{Jonathan Elmer}
\address{Middlesex University\\
The Burroughs, Hendon, London\\
NW4 4BT UK}
\email{j.elmer@mdx.ac.uk}

\date{\today}

\begin{document}

\maketitle

\begin{abstract} 
Let $G$ be a linear algebraic group acting linearly on a vector space (or more generally, an affine variety) $\vv$, and let $\kvg$ be the corresponding algebra of invariant polynomial functions. A separating set $S \subseteq \kvg$ is a set of polynomials with the property that for all $v,w \in \vv$, if there exists $f \in \kvg$ separating $v$ and $w$, then there exists $f \in S$ separating $v$ and $w$.

In this article we consider the action of $G = \SL_2(\C) \times \SL_2(\C)$ on the $\C$-vector space $\M_{2,2}^n$ of $n$-tuples of $2 \times 2$ matrices by multiplication on the left and the right. Minimal generating sets $S_n$ of $\C[\M_{2,2}^n]^G$ are known, and $|S_n| = \frac{1}{24}(n^4-6n^3+23n^2+6n)$. In recent work, Domokos \cite{DomokosSemi} showed that for all $n \geq 1$, $S_n$ is a minimal separating set by inclusion, i.e. that no proper subset of $S_n$ is a separating set. This does not necessarily mean that $S_n$ has minimum cardinality among all separating sets for $\C[\M_{2,2}^n]^G$. Our main result shows that any separating set for $\C[\M_{2,2}^n]^G$ has cardinality $\geq 5n-9$. In particular, there is no separating set of size $\dim(\C[\M_2^n]^G) = 4n-6$ for $n \geq 4$. Further, $S_4$ has indeed minimum cardinality as a separating set, but for $n \geq 5$ there may exist a smaller separating set than $S_n$. We also consider the action of $G= \SL_l(\C)$ on $\M_{l,n}$ by left multiplication. In that case the algebra of invariants has a minimum generating set of size $\binom{n}{l}$ (the $l \times l$ minors of a generic matrix) and dimension $ln-l^2+1$. We show that a separating set for $\C[\M_{l,n}]^G$ must have size at least $(2l-2)n-2(l^2-l)$. In particular, $\C[\M_{l,n}]^G$ does not contain a separating set of size $\dim(\C[\M_{l,n}]^G)$ for $l \geq 3$ and $n \geq l+2$. We include an interpretation of our results in terms of representations of quivers, and make a conjecture generalising the Skowronski-Weyman theorem.
\end{abstract}

\section{Introduction}
\subsection{Matrix semi-invariants}

Let $l,m \geq 1$ and let $\kk$ be a field. Denote by $\M_{l,m}$ the set of $l \times m$ matrices with coefficients in $\kk$. The group $G:= \SL_l(\kk) \times \SL_m(\kk)$ acts on $\M_{l,m}$ via the formula
\[(g,h) \cdot A = g A h^{-1}\]
where $g \in \SL_l(\kk), h \in \SL_m(\kk)$ and $A \in \M_{l,m}$. More generally we can consider the diagonal action of $G$ on the set $\M_{l,m}^n$ of $n$-tuples of $l \times m$ matrices.
Elements of $\M^n_{l,m}$ can be viewed as $n$-tuples $\bA = (A_1, A_2, \ldots, A_n)$ of $l \times m$ matrices, or as $l \times m$ matrices with elements in $\kk^n$. We call these $n$-matrices for short. For $(g,h) \in G$ we write
\begin{equation}\label{action}
(g,h) \cdot \bA := (g A_1 h^{-1},g A_2 h^{-1}, \ldots, g A_n h^{-1}).
\end{equation} 

Determining whether a pair $A, A'$ of $l \times m$ matrices lie in the same $G$-orbit is straightforward. For all $A \in \M_{l,m}$, one may find $(g,h) \in G$ such that $B:= (g,h) \cdot A$ has the following {\it reduced form:} for some $r \leq \min(l,m)$ (the rank of $A$), the entries $b_{ij}$ of $B$ are nonzero if and only if $i=j \leq r$. A pair $(A,A')$ of matrices in reduced form lie in the same $G$-orbit if and only if they have the same rank and their diagonal elements are the same up to order. However, if $l,m,n > 1$ then determining whether a pair of $n$-matrices lie in the same $G$-orbit is a very difficult problem.

There is an action of $\GL_n(\kk)$ on $\M_{l,m}^n$ which commutes with the action of $G$: namely, for $g \in \GL_n(\kk)$ and an $n$-matrix $\bA = ({\bf a}_{ij}) \in \M_{l,m}(\kk^n)$ we write $g \star \bA$ for the $n$-matrix whose $i,j$ entry is
\begin{equation}\label{commuting action} (g \star \bA)_{ij} = g({\bf a}_{ij}). \end{equation}

Now for $1 \leq i \leq l, 1 \leq j \leq m$ and $1 \leq k \leq n$, let $x_{ij}^{(k)}$ denote the linear functional $\M_{l,m}^n \rightarrow \kk$ which picks out the $i,j$th entry of of $A_k$, and introduce generic matrices
\[X_k:= \begin{pmatrix} x_{11}^{(k)} & x_{12}^{(k)} \ldots & x_{1m}^{(k)} \\ x_{21}^{(k)} & x_{22}^{(k)} \ldots & x_{2m}^{(k)} \\
\vdots & \vdots & \vdots \\ x_{l1}^{(k)} & x_{l2}^{(k)} \ldots & x_{lm}^{(k)} 
 \end{pmatrix}.\]
Then we have \[\kk[\M_{l,m}^n] = \kk[x_{ij}^{(k)}: i = 1, \ldots, l, j = 1, \ldots, m, k = 1, \ldots, n].\]

The action of $G$ on $\M_{l,m}^n$ induces an action of $G$ on $\kk[\M_{l,m}^n]$ by algebra automorphisms: we define
\[(g \cdot f)(\bA) = f(g^{-1} \cdot \bA)\] for all $g \in G$, $f \in \kk[\M_{l,m}^n]$ and $\bA \in \M_{l,m}^n$. The set $\kk[\M_{l,m}^n]^G$ of fixed points of this action forms a $\kk$-subalgebra. Elements of $\kk[\M_{l,m}^n]^G$ are called \emph{matrix semi-invariants.} The algebra $\C[\M_{l,m}^n]^G$ has been intensely studied over the years. A minimal generating set is known for arbitrary $n$ only in the cases $l=1$, $m=1$, $l=m=2$, or $l=m=3$. In the case $m=1$, it is clear that $M_{l,1}^n$ is isomorphic as a $SL_l(\C)$-module to $\M_{l,n}$ with $SL_l(\C)$ acting on the latter via left multiplication. Then the algebra of invariants is well known, see e.g. \cite[Theorem~4.4.4]{DerksenKemper}:

\begin{prop}\label{fftsln} Consider the action of $\SL_l(\C)$ on $\M_{l,n}$ via left multiplication. Then the ring of invariants is
\[\C[\M_{l,n}]^{\SL_l(\C)} = \C[\det(X_{k_1,k_2,\ldots,k_l}): 1 \leq k_1<k_2<\cdots<k_l \leq n]\]
where $X_{k_1,k_2, \ldots, k_l}$ denotes the submatrix of $X$ consisting of the columns labelled by $k_1, k_2, \ldots, k_l$. Note that if $n<l$ the ring of invariants is trivial.  
\end{prop}

Clearly the case $l=1$ is analogous to the above. In the case $l=m=2$, a minimal set of homogeneous generators is known:
\begin{prop}[Domokos \cite{DomokosPoincare}]\label{fft2}

The following is a minimal set of homogeneous generators for $\C[\M_{2,2}^n]^G$
\begin{itemize}
\item $\det(X_i): 1 \leq i \leq n$;
\item $\langle X_i | X_j \rangle:= \Tr(X_i)\Tr(X_j) - \Tr(X_iX_j): 1 \leq i<j \leq n$;
\item $\xi(X_i,X_j,X_k,X_l): 1 \leq i<j<k<l \leq n.$
\end{itemize}

Here $\xi(X_i,X_j,X_k,X_l)$ is the coefficient of $a_ia_ja_ka_l$ in the determinant
\[\begin{vmatrix} a_iX_i & a_jX_j\\ a_kX_k & a_lX_l \end{vmatrix} \in \C[\M_2^n][a_i,a_j,a_k,a_l].\]
\end{prop}

We note that a finite set of generating invariants is also known in the case $l=m=3$, see \cite{Lopatin3x3}.

\subsection{Separating Invariants and Dimension}

Now consider a more general situation in which a linear algebraic group $G$ defined over $\kk$ acts linearly on an affine $\kk$-variety $\vv$. Let $\kv$ denote the algebra of polynomial functions on $\vv$. Then $G$ acts on $\kvg$ according to the formula
\begin{equation}\label{action} (g \cdot f) (v) = f(g^{-1} \cdot v).\end{equation} We denote by $\kvg$ the subalgebra of $\kv$ fixed by this action. Further define the \emph{Nullcone} of $G$ on $\vv$:
\begin{equation} \mathcal{N}_{G,\vv} = \{v \in\vv: f(v)=0 \ \text{for all} \ f \in \kvg\}.
\end{equation}

Assume that $G$ is reductive. Then we have the following characterisation of the Krull dimension of $\kvg$:

\begin{prop} Let $f_1, f_2, \ldots, f_d \in \kvg$. Then the following are equivalent:
\begin{enumerate}
\item $\kvg$ is finitely generated over $\kk[f_1,f_2, \ldots, f_d]$;
\item $\mathcal{N}_{G,\vv} = V(f_1,f_2, \ldots, f_d)$.
\end{enumerate}
Moreover, the minimum $d>0$ such that there exist $f_1, f_2, \ldots, f_d \in \kvg$ satisfying the above is equal to the Krull dimension of $\kvg$.
\end{prop}

A set $f_1, f_2, \ldots, f_d \in \kvg$ with the properties above is sometimes called a \emph{zero-separating set}, since for any $v \in \vv$ we have that if there exists $f \in \kvg$ with $f(v) \neq 0$ then $f_i(v) \neq 0$ for some $1 \leq i \leq d$. A zero-separating set with minimum cardinality (equivalently, an algebraically independent zero-separating set) consisting of homogeneous polynomials is called a \emph{homogeneous system of parameters}.

Suppose further that there exists no non-trivial character $G \rightarrow \C^*$. Then we have
\begin{equation}\label{krulldim}
\dim(\kvg) = \dim(\vv)-\dim(G)+\min_{v \in \vv} \dim(G_v),
\end{equation}
see \cite[Corollary~2.3]{Elmer2x2} for proof.

\begin{eg} Consider the situation of Proposition \ref{fftsln} and suppose $n \geq l$. Then we have
\[\dim(\kvg) = ln - l^2 + 1.\]
It follows that one may find a homogeneous system of parameters for this algebra of invariants with cardinality $ln-l^2+1$. In fact, since the invariants $\det(X_{k_1,\ldots,k_l})$ all have the same degree, it follows from the proof of the Noether Normalisation Lemma (see \cite[Lemma~2.4.7]{DerksenKemper} that one may find a homogeneous system of parameters with this cardinality consisting of linear combinations of these $\det(X_{k_1,\ldots,k_l})$.
\end{eg}

\begin{eg}The dimension of $\C[\M^n_{2,2}]^G$ for $n\geq 3$ is $\dim(\M_2^n) - \dim(G) = 4n-6$. This follows from Equation \eqref{krulldim} because there exist $3$-matrices whose stabiliser in $G$ is the finite group $(\pm I, \pm I)$. Contrastingly, $\dim (\C[\M^2_{2,2}]^G) = 8-6+1=3$, since every $2$-matrix has at least a 1-dimensional stabiliser, and $\dim(\C[\M_{2,2}]^G) = 4-6+3=1$ since the stabiliser of any matrix has dimension at least 3.
\end{eg}

If $f \in \kvg$ and $f(v) \neq f(w)$ we say that $f$ {\it separates} $v$ and $w$. We say that $v$ and $w$ are {\it separated by invariants} if there exists an invariant separating $v$ and $w$.  In case $G$ is reductive, we have that $f(v) = f(w)$ for all $f \in \kvg$ if and only if $\overline{Gv} \cap \overline{Gw} \neq \emptyset$ where the bar denotes closure in the Zariski  topology, see  \cite[Corollary~6.1]{Dolgachev}. In particular, the invariants separate the orbits if $G$ is a finite group.

One can in principle separate orbits whenever one can find an explicit generating set for $\kvg$, but this is an extremely difficult problem in general. For this reason, Derksen and Kemper introduced the following in 2002 \cite[Definition~2.3.8]{DerksenKemper}:

\begin{Def}\label{sepdef}
Let $S \subseteq \kvg$. We say $S$ is a {\it separating set} for $\kvg$ if the following holds for all $v, w \in \vv$:
\begin{equation*} s(v) = s(w) \ \text{for all} \ s \in S \Leftrightarrow f(v) = f(w)  \ \text{for all} \ f \in \kvg. \end{equation*}
\end{Def}

Separating sets of invariants have been an area of much recent interest. In general they have nicer properties and are easier to construct than generating sets. For example, if $G$ is a finite group acting on a vector space $V$, then the set of invariants of degree $\leq |G|$ is a separating set \cite[Theorem~3.9.14]{DerksenKemper}. This is also true for generating invariants if $\chr(\kk) = 0$ \cite{FleischmannNoetherBound}, \cite{FogartyNoetherBound} but fails for generating invariants in the modular case. Separating sets for the rings of invariants $\kk[V]^{C_p}$, where $\kk$ is a field of characteristic $p$ and $C_p$ the cyclic group of order $p$ and $V$ is indecomposable were constructed in \cite{SezerCyclic}. Corresponding sets of generating invariants are known only when $\dim(V) \leq 10$ \cite{WehlauCyclicViaClassical}. For the (non-reductive) linear algebraic group $\Ga$ of a field of characteristic zero, separating sets for $\kk[V]^{\Ga}$ for arbitary indecomposable linear representations $V$ were constructed in \cite{ElmerKohls}. These results were extended to decomposable representations in \cite{DufresneElmerSezer}. Even for indecomposable representations, generating sets are known only where $\dim(V) \leq 8$ \cite{Bedratyuk7}.  Finally, for an arbitrary (i.e. non-linear) $\Ga$-variety $\vv$, the algebra of invariants $\kk[\vv]^{\Ga}$ may not be finitely generated, but it is known that there must exist a finite separating set \cite{KemperSeparating} and finite separating sets have been constructed for many examples where $\kk[\vv]^{\Ga}$ is infinitely generated \cite{DufresneKohlsFiniteSep, DufresneKohlsSepVar}. 

Evidently a separating set must be a zero-separating set. Therefore, if $G$ is reductive, the Krull dimension of $\kvg$ is a lower bound for the size of any separating set. In fact this remains true for arbitrary linear algebraic groups. Let $S$ be a separating set for $\kvg$ consisting of homogeneous polynomials. The subalgebra $\kk[S]$ of $\kvg$ generated by $S$ is called a {\it separating algebra}. By \cite[Proposition~3.2.3]{DufresnePhD}, the quotient fields of $\kk[S]$ and $\kvg$ have the same transcendence degree over $\kk$. Then by \cite[Proposition~2.3(b)]{Giral} we get that $\dim(\kk[S]) = \dim(\kvg)$. Consequently, the size of a separating set is bounded below by the dimension of $\kvg$. 

A separating set whose size equals the dimension of $\kvg$ is sometimes called a polynomial separating set, because it necessarily generates a polynomial subalgebra of $\kvg$. On the other hand, there always exists a separating set of size $\leq 2\dim(\kvg)+1$, albeit such a separating set may necessarily contain non-homogeneous polynomials; see \cite[Theorem~5.3]{KamkeKemper} for a proof. In the present article we shall say a separating set $S$ is \emph{minimal} if no proper subset of $S$ is separating; and \emph{minimum} if it has the smallest cardinality among all separating sets.

%separating variety section goes here?
\subsection{The separating variety}

The main tool in our proofs will be the separating variety. This was introduced by Kemper in \cite{KemperCompRed}:

\begin{Def}\[ \mathcal{S}_{G,\vv} = \{(v, w) \in \vv^2: f(v)= f(w) \ \text{for all} \ f \in \kvg \} .\]
\end{Def}

In other words, the separating variety is the subvariety of $\vv^2$ consisting of pairs of points which are not separated by any invariant.  More scheme-theoretically, one may equivalently define
\begin{equation}\label{sepscheme} \mathcal{S}_{G,\vv} = (\vv \times_{\vv/G} \vv)_{\text{red}},
\end{equation}
i.e. $\mathcal{S}_{G,\vv}$ is the unique reduced scheme whose underlying variety is the fibre product $(\vv \times_{\vv/G} \vv)$. Given this observation, one might well expect that
\[\dim(S_{G,\vv}) = 2 \dim(\vv) - \dim(\kvg).\]
This is not the case in general however, as we will see in the next section.

We define $I_{G,\vv}$ to be the ideal of $\kk[\vv^2]$ consisting of the polynomial functions which vanish on $\mathcal{S}_{G,\vv}$. Clearly this is a radical ideal. Then a separating set can be characterised as a subset $S \subseteq \kvg$ which cuts out the separating variety in $\vv^2$, in other words (see \cite[Theorem~2.1]{DufresneSeparating}):

\begin{prop}\label{radversion}
$S \subseteq \kk[\vv]^G$ is a separating set if and only if 
\[V_{\vv^2}(\delta(S)) = \mathcal{S}_{G,\vv}.\]
\end{prop}
where  $\delta: \kk[\vv] \rightarrow \kk[\vv^2] = \kk[\vv] \otimes \kk[\vv]$ is defined by
\[\delta(f) = 1 \otimes f - f \otimes 1.\] 

Equivalently, via the Nullstellensatz, $S$ is a separating if and only if
\[\sqrt{(\delta(S)} = {I_{G,\vv}}.\]
%radical not needed on LHS
Consequently the size of a separating set in $\kvg$ is bounded below by the minimum number of generators of $I_{G,\vv}$ up to radical, that is, the minimum number of elements generating any ideal whose radical is $I_{G,\vv}$ (this is sometimes called the {\it arithmetic rank} of $I_{G,\vv}$). We then find, using Krull's height theorem, (see e.g. \cite[Theorem~10.2]{Eisenbud}) that:
\begin{prop}\label{krullbound}
Let $S \subseteq \kk[\vv]^G$ by a separating set. Then $|S| \geq \codim_{\vv^2}(\cc)$ for all irreducible components $\cc$ of $\mathcal{S}_{G,\vv}$.
\end{prop}

Therefore, in order to use Proposition \ref{krullbound} above to find lower bounds for separating sets, we must decompose $\mathcal{S}_{G,\vv}$ into irreducible components. As a first step, we observe that the separating variety contains the following subvariety, which we call the {\it graph} of the action:
\begin{Def} \[\Gamma_{G,\vv} = \{(v,gv): v \in \vv, g \in G \}.\]
\end{Def}
If $G$ is connected and reductive, then $\overline{\Gamma_{G,\vv}}$ is an irreducible component of $\mathcal{S}_{G,\vv}$. Its dimension is easily seen to equal $\dim(\vv) + \dim(G) - \min\{\dim(G_v): v \in \vv\}$, and we note that in case there is no nontrivial character $G \rightarrow \C^*$ that
\begin{equation}\label{graphdim} \dim(\overline{\Gamma_{G,\vv}}) = 2 \dim(\vv) - \dim(\kvg). \end{equation}
In particular this implies $\dim(S_{G,\vv}) \geq  2 \dim(\vv) - \dim(\kvg),$ but we do not have equality in general. Furthermore,  the separating variety may have extra components of smaller dimension. These components are an obstruction to the existence of small separating sets.

We may obtain a first step towards decomposing $\mathcal{S}_{G,\vv}$ as follows: notice that if $(v,v') \in \mathcal{S}_{G,\vv}$ then we have $(v,v') \in \overline{\Gamma_{G,\vv}}$ unless neither $G \cdot v$ nor $G \cdot v'$ is closed. Now we recall the {\it Hilbert-Mumford criterion} \cite{Mumford}: given $v \in \vv$ and a 1-parameter subgroup $\lambda: \C^* \rightarrow G$ we define $w(\lambda, v)$ to be the unique integer $d$ such that 
\[\lim_{t \rightarrow 0} \lambda(t) \cdot t^d v \]
exists and is non-zero. $v \in \vv$ is called {\it stable} if $w(\lambda, v)>0$ for all 1-parameter subgroups $\lambda: \C^* \rightarrow G$, and  {\it semi-stable} if $w(\lambda, v) \geq 0$ for all 1-parameter subgroups $\lambda: \C^* \rightarrow G$. A point $v \in\vv$ is stable if and only if $G \cdot v$ is closed and the dimension of $\overline{G \cdot v}$ is maximal. The set $\vv_s$ of stable points forms an open subset of $\vv$. A point is called {\it non-stable} if it is not stable (the more natural term unstable usually means not semi-stable). It follows that the set $\uu:= \vv \setminus \vv_s$ of non-stable points forms a closed subset of $\vv$. We thus obtain a decomposition

\begin{equation}\label{firstdecomp} \mathcal{S}_{G,\vv} = \overline{\Gamma_{G,\vv}} \cup \mathcal{S}_{G,\vv,\uu}
\end{equation}
where $\mathcal{S}_{G,\vv,\uu}:= \uu^2 \cap \mathcal{S}_{G,\vv}$. Since both $\uu^2$ and $\mathcal{S}_{G,\vv}$ are closed irreducible subsets of $\vv^2$ it follows that $\mathcal{S}_{G,\vv,\uu}$ is closed. It may or may not be irreducible; this depends on the action of $G$ on $\vv$.

Stronger obstructions may be obtained by taking a closer look at the geometry of $\mathcal{S}_{G,\vv}$. Recall that a Noetherian toplogical space $\vv$ is said to be {\it connected in dimension $k$} if the following holds: for each closed subvariety $Z \subseteq \vv$ with dimension $< k$, the complement $\vv \setminus Z$ is connected. If the same holds for all $Z \subseteq \vv$ with $\codim_{\vv}(Z)  > k$, we say that $\vv$ is {\it connected in codimension $k$}. Note that if $\vv$  is equidimensional, or all irreducible components of $\vv$ intersect nontrivially then we have $\dim(Z) = \dim(\vv)- \codim_{\vv}(Z)$; consequently $\vv$ is connected in dimension $k$ if and only if it is connected in codimension $\dim(\vv)-k$. The following result was proved is \cite{Elmer2x2}:

\begin{prop}\label{grothbound} Suppose $\mathcal{S}_{G,\vv}$ is not connected in codimension $k$, and let $S \subseteq \kvg$ be a separating set. Suppose further that all irreducible components of $\mathcal{S}_{G,\vv}$ intersect nontrivially, and that there does not exist a non-trivial character $G \rightarrow \kk^*$. Then $|S| \geq \dim(\kvg) + k$.
\end{prop}

\subsection{Statement of results}
 
We return to the notation and setting of section 1.1. Our goal is to desribe the separating variety and use this description to find a lower bound for the size of a separating set. We first consider the case $m=1$:

\begin{Theorem} 
\label{main1l} Let $S \subseteq \C[\M_{l,1}^n]^G$ be a separating set, and suppose $l \leq n$.  Then $|S| \geq (2l-2)n - 2(l^2-l)$.
\end{Theorem}

The dimension of the algebra of invariants $\C[\M^n_{l,1}]^G$ is $ln-l^2+1$. Thus, we see that $\C[\M^n_{l,1}]^G$ does not contain a polynomial separating set for $l \geq 3$ and $n \geq l+2$. On the other hand, the minimum generating set for $\C[\M^n_{l,1}]^G$ given in Proposition \ref{fftsln} has cardinality $\binom{n}{l}$, so we see that for $l \geq 3$ and $n \geq l+2$ this may not be a minimum separating set. The question of constructing minimum separating sets for this action will be taken up in a further article.
%cite elmerreimers when preprint written

Our main results concern the case $l=m=2$. Recently, Domokos proved that the generating set given in Proposition \ref{fft2} is a minimal separating set. Note that this does not necessarily mean it is a minimum separating set. In the present article we show:

\begin{Theorem}\label{main22} Let $S \subseteq \C[\M_{2,2}^n]^G$ be a separating set. Then $|S| \geq 5n-9$.
\end{Theorem}

Recall that the Krull dimension of $\C[\M_{2,2}^n]^G$ is $4n-6$ for $n \geq 3$. Therefore our results show that $\C[\M_{2,2}^n]^G$ does not have a polynomial separating set for $n \geq 4$. 

\begin{rem} In \cite{Elmer2x2} we showed that in the situation above one always has $|S| \geq 5n-10$. While the effort in this article results only in an improvement of 1, this closes a crucial gap in that we were not able to conclude whether or not a polynomial separating set for $n=4$ existed until now.
\end{rem}
These results are proved by examining in detail the connectivity structure of the separating variety. Since these results are likely of independent interest, we state them here:

\begin{Theorem}\label{sepvar1l} Let $\vv = \M_{l,1}^n$, $G = \SL_l(\C)$. 
\begin{itemize}
\item[(a)] Suppose $l \geq 3$. The separating variety $\mathcal{S}_{G,\vv}$ has two components of dimensions $ln+l^2-1$ and $2(n+1)(l-1)$ respectively, which intersect in a closed subvariety of dimension $\leq ln+l^2-2$ (with equality if $l=3$).

\item[(b)] Suppose $l=2$. Then $\mathcal{S}_{G,\vv}$ has just one irreducible component.
\end{itemize}
\end{Theorem}

We note that for $l\geq 3$ and $n \geq l+2$ the separating variety has dimension $2(l-1)(n+1)> ln+l^2-1 = 2 \dim(\vv) - \dim(\kvg)$.

\begin{Theorem}\label{sepvarstructure} Let $\vv = \M_{2,2}^n$, $G = \SL_2(\C) \times \SL_2(\C)$ and suppose $n \geq 4$. Then the separating variety $\mathcal{S}_{G,\vv}$ has three components of dimensions $4n+6$ (the graph closure), $4n+5$ and $4n+5$ respectively. Each component of dimension $4n+5$ intersects the graph closure in a closed subvariety of dimension $3n+8$. The intersection of the two components of dimension $4n+5$ has dimension $3n+6$ and is completely contained in the graph closure.
\end{Theorem}

\subsection{Structure of paper} The second section of this paper focuses on the left action of $\SL_l(\C)$ on $\M_{l,n}(\C)$, and the goal is the proof of Theorem \ref{main1l}. The third section focuses on the case $l=m=2$, i.e. of the left-right action on $2 \times 2$ matrices, and the goal is the proof of Theorem \ref{main22}. In the final section we interpret some results of this article and \cite{Elmer2x2} in terms of representations of quivers, and make a conjecture generalising the celebrated Skowronski-Weyman theorem.

\begin{ack} This research was partially funded by the EPSRC small grant scheme, ref: EP/W001624/1. The author thanks the research council for their support. It was partially written while visiting Prof. Harm Derksen at Northeastern University, and the author wishes to thank Prof. Derksen for his hospitality and some helpful suggestions.
\end{ack}

\section{Left Invariants}
In this section we consider the case $m=1$. We can view elements of $\vv$ as $l \times n$ matrices, or as $l$-tuples of row vectors having length $n$. $G = \SL_{l}(\C)$ acts by left-multiplication on elements of $\vv$ viewed as matrices, and there is a commuting action of $\GL_{n}(\C)$ on $\vv$ by right--multiplication. A generic element of $\vv$ will be written as
\[A = (\ba_1,\ba_2, \ldots, \ba_l)^t\] where for each $i=1, \ldots, n$ we write 
\[\ba_i = (a_{ij}: j=1,\ldots, n).\]
We let $X$ denote the generic $l \times n$ matrix of coordinate functions on $\vv$, so that $x_{ij}(A) = a_{ij}$.
\iffalse
For any subset $Q \subseteq {1, \ldots, n}$ of size $l$ we let $A_Q$ denote the $l \times l$ submatrix of $A$ consisting of the columns indexed by ordered elements of $Q$, and we define $X_Q$ similarly, so that $\det(X_Q)(A) = \det(A_Q)$. Then
\[\C[\vv]^G = \C[X_Q: Q \subseteq \{1, \ldots, n\}: |Q| = l].\] 
\fi %we don't use this notation
We assume $n \geq l$, since otherwise $\C[\vv]^G = \C$.

We begin by determining the non-stable points in $\vv$. Every 1-parameter subgroup of $G$ is contained in a maximal torus, and the maximal tori in $G$ are all conjugate. A 1-parameter subgroup of the diagonal group $T$ takes the form

\[\{\lambda_{\br}(t):= \begin{pmatrix}t^{r_1} & 0 & \cdots & 0\\
 0 & t^{r_2} & \cdots & 0 \\
 \vdots & \vdots & \cdots & \vdots \\
0 & 0 & \cdots & t^{r_l} \end{pmatrix}: t \in \C^* \}\] where $\br \in \Z^l$ with $\sum_{i=1}^l r_i = 0$. 

Therefore $A$ is non-stable for $T$ if and only if $A$ has a row of zeroes. It follows that $A$ is non-stable for $G$ if and only if $\rk(A) < l$. Note that all such matrices lie in $\mathcal{N}_{G,\vv} \subseteq \mathcal{S}_{G,\vv}$. Applying equation \eqref{firstdecomp} we obtain:

\begin{equation}\label{leftdecomp} \mathcal{S}_{G,\vv} = \overline{\Gamma_{G,\vv}} \cup \mathcal{N}_{G,\vv}^2.
\end{equation}

Clearly $\mathcal{N}_{G,\vv}$ is closed and irreducible, so the same is true of $\mathcal{N}_{G,\vv}^2$. Therefore the separating variety $\mathcal{S}_{G,\vv}$ has at most two components. The dimension of $\overline{\Gamma_{G,\vv}}$ is $$2 \dim(\vv) - \dim(\kvg) = ln+l^2-1.$$ On the other hand, $$\dim(\mathcal{N}_{G,\vv}) = (l-1)(n+1).$$ Therefore if $l \geq 3$ and $n \geq l+2$ we have
\[\dim(\mathcal{N}_{G,\vv}^2) = 2(l-1)(n+1) > ln+l^2-1 = \dim(\overline{\Gamma_{G,\vv}}),\]
and therefore
\[ \dim(\mathcal{S}_{G,\vv}) = 2(l-1)(n+1) > ln+l^2-1 = 2\dim(\vv) - \dim(\kvg)).\]

In order to find lower bounds for the size of separating sets, we would need to determine $\overline{\Gamma_{G,\vv}} \cap \mathcal{N}_{G,\vv}^2$. For $A, A' \in \vv = \M_{l,n}$, $A|A' \in \M_{2l,n}$ denotes the $2l \times n$ matrix obtained by stacking $A$ and $A'$.  

\begin{Lemma}\label{rankbound} Let $(A, A') \in \mathcal{N}_{G,\vv}^2$. Then $(A,A') \in \overline{\Gamma_{G,\vv}}$ only if $\rk(A|A') \leq l$.
\end{Lemma}

\begin{proof} Note that in general we must have $\rk(A), \rk(A') \leq l-1$ so $\rk(A|A') \leq 2l-2$. Therefore, there is nothing to prove if $l=2$; we may therefore assume $l \geq 3$.

Let $(A,A') \in \mathcal{N}_{G,\vv}^2$. Since $\overline{\Gamma_{G,\vv}}$ is preserved by the action of $G^2$, we may assume that $A$ and $A'$ are in row echelon form; in particular their bottom rows are both zero. 
 
Now suppose in addition that $(A,A') \in \overline{\Gamma_{G,\vv}}$. Then there exist functions $$g = (g_{ij}: i,j=1, \ldots, l): \C^* \rightarrow G$$, $$A = (\ba_1,\ldots, \ba_l)^t: \C^* \rightarrow \vv$$ and $$A' = (\ba'_1,\ldots, \ba'_l)^t: \C^* \rightarrow \vv$$ such that 
\[\lim_{t \rightarrow 0}A(t) = A, \lim_{t \rightarrow 0}(A'(t)) = A', g(t) \cdot A(t) = A'(t) \ \text{for all} \ t \in \C^*\] where we abuse notation by using the same letter for a function and its limit. Therefore, for all $t \in \C^*$ we have
\[ A(t)|A'(t) = \begin{pmatrix} \ba_1(t) \\ \ba_2(t) \\ \vdots \\ \ba_l(t) \\ \sum_{j=1}^l g_{1j}(t) \ba_j(t) \\  \sum_{j=1}^l g_{2j}(t) \ba_j(t) \\ \vdots \\ \sum_{j=1}^l g_{lj}(t) \ba_j(t)  \end{pmatrix}.\]
Clearly the rank of $(A(t)|A'(t))$ is at most $l$, since its $2l$ rows can be expressed as linear combinations of $l$ elements of $\C^n$. Since the set of matrices with rank bounded above by a given number is closed, it follows that $$(A|A') = \lim_{t \rightarrow 0}(A(t)|A'(t))$$ also has rank at most $l$.
\end{proof}

We have a partial converse to Lemma \ref{rankbound}:

\begin{prop} Suppose $l=2$ or $l=3$. Let $(A, A') \in \mathcal{N}_{G,\vv}^2$. Then $(A,A') \in \overline{\Gamma_{G,\vv}}$ if $\rk(A|A') \leq l$.
\end{prop}

\begin{proof} We deal with the two cases separately. Since $\overline{\Gamma_{G,\vv}}$ and $\mathcal{N}_{G,\vv}^2$ are preserved by the action of $G^2$ we may again assume $A$ and $A'$ are in row echelon form. Suppose $l=2$ and note that the condition $\rk(A|A') \leq 2$ is vacuous. Write $$A = \begin{pmatrix} \mathbf{a} \\ 0 \end{pmatrix},  A'= \begin{pmatrix} \mathbf{a}' \\ 0 \end{pmatrix},$$
and note that $(A,A') = \lim_{t \rightarrow 0} (A(t),A'(t))$ where
\[A(t) = \begin{pmatrix} \mathbf{a} \\ t(\mathbf{a}'-\mathbf{a}) \end{pmatrix}, A'(t) = \begin{pmatrix} \mathbf{a}' \\ t(\mathbf{a}'-\mathbf{a}) \end{pmatrix} = \begin{pmatrix} 1& t^{-1}  \\ 0&1 \end{pmatrix}A(t).\]

Now suppose $l=3$. Write  $$A = \begin{pmatrix} \mathbf{a}_1 \\ \ba_2 \\ 0 \end{pmatrix},  A'= \begin{pmatrix} \mathbf{a}_1' \\ \ba'_2\\ 0 \end{pmatrix}.$$ 

Suppose first that $\dim(\langle \ba'_1,\ba'_2 \rangle) < 2$. Then there exist $\lambda'_1, \lambda'_2 \in \C$, not both zero with
\[\lambda'_1 \ba'_1 + \lambda'_2 \ba'_2 = 0.\]
As $G$ contains the matrix \[P:= \begin{pmatrix} 0 & 1 & 0 \\ 1 & 0 & 0 \\ 0 & 0 & -1 \end{pmatrix}\] we may assume without loss of generality that $\lambda'_2 \neq 0$. Then
\[\begin{pmatrix} 1 & 0 & 0 \\ \lambda'_1 & \lambda'_2 & 0 \\ 0 & 0 & {\lambda'_2}^{-1} \end{pmatrix} A' = \begin{pmatrix} \ba'_1 \\ 0 \\ 0 \end{pmatrix} \in G \cdot A'\] and 
\[\left(\begin{pmatrix} \ba_1 \\ \ba_2 \\ 0  \end{pmatrix},  \begin{pmatrix} \ba'_1 \\ 0 \\ 0 \end{pmatrix}\right) =
\lim_{t \rightarrow 0}\left(\begin{pmatrix} \ba_1 \\  \ba_2 \\ t^2(\ba'_1 - \ba)  \end{pmatrix},  \begin{pmatrix} 1 & 0 & t^{-2} \\ 0 & t & 0  \\ 0 & 0 & t^{-1}  \end{pmatrix}\begin{pmatrix} \ba_1 \\ \ba_2 \\   t^2(\ba'_1 - \ba)\end{pmatrix} \right)  \in \overline{\Gamma_{G,\vv}}.\] 

Similarly we can show $(A,A') \in \overline{\Gamma_{G,\vv}}$ if $\dim(\langle \ba_1,\ba_2 \rangle) < 2$, so we may assume that $\dim(\langle \ba_1,\ba_2 \rangle) = 2 =\dim(\langle \ba'_1,\ba'_2 \rangle)$.

Now since $rk(A|A') \leq 3$ there exist $\lambda_1, \lambda_2, \lambda'_1, \lambda'_2 \in\C$, not all zero, such that
\[\bb:= \lambda_1 \ba_1 + \lambda_2 \ba_2 = \lambda'_1\ \ba'_1 + \lambda'_2 \ba'_2,\] and by the previous two paragraphs, we may assume that $\lambda_2 \neq 0$ and $\lambda'_2 \neq 0$. We obtain that
\[
\begin{pmatrix} \ba_1 \\ \bb \\ 0 \end{pmatrix} \in G \cdot A, \begin{pmatrix} \ba'_1 \\ \bb \\ 0 \end{pmatrix} \in G \cdot A' ,\]
and
\[\left(\begin{pmatrix} \ba_1 \\ \bb \\ 0 \end{pmatrix}, \begin{pmatrix} \ba'_1 \\ \bb \\ 0 \end{pmatrix}\right) = \lim_{t \rightarrow 0}  \left(\begin{pmatrix} \ba_1 \\ \bb \\ t(\ba'-\ba) \end{pmatrix}, \begin{pmatrix}1 & 0 & t^{-1} \\ 0 & 1 & 0\\ 0 & 0 & 1 \end{pmatrix}\begin{pmatrix} \ba_1 \\ \bb \\ t(\ba'-\ba) \end{pmatrix}\right) \in \overline{\Gamma_{G,\vv}}.\]
\end{proof}

In the case $l=2$, we conclude that the $\mathcal{N}_{G,\vv} = \overline{\Gamma_{G,\vv}}$, and that our methods do not allow us to improve the lower bound $\dim(\kvg) = 2n-3$ for the size of a separating set. Note however that for any separating set $S$ we certainly have
\[|S| \geq 2n-3 > 2n-4 = (2l-2)n-2(l^2-l)\] so the claimed bound holds.

Lemma \ref{rankbound} showed that, for $l \geq 3$, the intersection $\overline{\Gamma_{G,\vv}} \cap \mathcal{N}_{G,\vv}^2$ is contained in the variety $\mathcal{Z} \subset \M_{2l,n}$ consisting of matrices whose first $l$ rows span a subspace of $\C^n$ with dimension $<l$, whose last $l$ rows span a subspace of $\C^n$ with dimension $<l$, and whose $2l$ rows together span a subspace of dimension $\leq l$. We can compute the dimension of this variety. Let $Z \in \mathcal{Z}$ and write $\ba_1, \ba_2, \ldots, \ba_l$, $\bb_1, \bb_2, \ldots, \bb_l$ for its rows.  If we want to choose $Z \in \mathcal{Z}$, we have free choice of $\ba_1, \ldots, \ba_{l-1}$, giving $(l-1)n$ choices. Then we must choose $\ba_l \in \langle \ba_1, \ldots, \ba_{l-1} \rangle$, giving $l-1$ choices. We have free choice of $\bb_1$, but must then choose $\bb_2, \bb_3, \ldots, \bb_{l-1} \in \langle \ba_1, \ldots, \ba_{l-1}, \bb_1 \rangle$ giving $l$ choices for each. Finally we must choose $\bb_l \in \langle \bb_1, \ldots, \bb_{l-1} \rangle$, which happily guarantees that $\bb_l \in \langle \ba_1, \ldots, \ba_{l-1}, \bb_1 \rangle$ at the same time, and gives us $l-1$ further choices. Thus, the dimension of $\mathcal{Z}$ is

\[(l-1)n+(l-1)+n+(l-1)l+l-1 = ln+l^2-2.\]

In the case $l=3$, $n>3$, the separating variety has two components, $\overline{\Gamma_{G,\vv}}$ with dimension $2 - \dim(\kvg) = 3n+8$ and $\mathcal{N}_{G,\vv}^2$ with dimension $4n+4$. For $n=4$ these components have the same dimension $3n+8$, and
for $n \geq 5$ the component $\overline{\Gamma_{G,\vv}}$ is the smallest. Its codimension in $\vv^2$ is $3n-8$. These two components intersect in a subvariety of dimension $3n+7$. Thus, the separating variety is not connected in codimension
\[(4n+4)-(3n+8) = n-4\] and by Proposition \ref{grothbound} we have that any separating set for $\C[\vv]^G$ has cardinality at least
\[ 3n-8+(n-4) = 4n-12.\]

The case $l \geq 4$, $n \geq l+1$, is similar, in that there are two components of dimensions $2(l-1)(n+1)$ and $ln+l^2-1$ intersecting in a subvariety of dimension at most $ln+l^2-1$. So the separating variety is not connected in codimension $k$ for some 
\[k \leq 2(l-1)(n+1)-(ln+l^2-1) = ln-l^2+2l-2n\] and by Proposition \ref{grothbound} we have that any separating set for $\C[\vv]^G$ has cardinality at least
\[ ln-l^2+1+(ln-l^2+2l-2n) = (2l-2)n-2(l^2-l)\] as claimed. This completes the proof of Theorems \ref{main1l} and \ref{sepvar1l}.

\section{$2 \times 2$ Matrix semi-invariants}

In this section we speciallise to the case $l=m=2$. Thus, we set $\vv:= \M_{2,2}^n$ for some $n>0$ and $G = \SL_2(\C) \times \SL_2(\C)$. 
A generic element $\bA \in \vv$ will be written as $(A_1,A_2,\ldots,A_n)$ where
\[A_i = \begin{pmatrix} a_i & b_i \\ c_i & d_i \end{pmatrix}.\]
A generic element $(\bA,\bA') \in \vv \times \vv$ will be written with $\bA$ as above and $\bA' = (A'_1,A'_2,\ldots,A'_n)$ where
\[A'_i = \begin{pmatrix} a'_i & b'_i \\ c'_i & d'_i \end{pmatrix}.\]
Throughout we write $g \cdot \bA$ for the action of $g = (g_1,g_2) \in G$ on $\bA \in \vv$. We use the notation $h \star \bA$ for the commuting action of $h \in \GL_n$.

Our aim is to use \eqref{firstdecomp} to decompose $\mathcal{S}_{G,\vv}$. We begin by identifying non-stable points:

\begin{Lemma} Let $\bA \in \vv$. Then $A$ is not stable if and only there exists $g \in G$ such that $g \cdot A_i$ is simultaneously upper-triangular for all $i=1 \ldots, n$.
\end{Lemma}

\begin{proof} Every 1-parameter subgroup in $\SL_2(\C)$ is a maximal torus. Since the maximal tori are all conjugate in $\SL_2(\C)$, each is conjugate to the diagonal group
\[T:= \{\begin{pmatrix} t & 0 \\ 0 & t^{-1} \end{pmatrix}: t \in \C^*\}.\]
It follows that every 1-parameter subgroup in $G$ is conjugate to a group of the form
\[T_{r,s}: \{\lambda_{r,s}(t):=   \left(\begin{pmatrix} t^r & 0 \\ 0 & t^{-r} \end{pmatrix}, \begin{pmatrix} t^s & 0 \\ 0 & t^{-s} \end{pmatrix}   \right): t \in \C^* \}\] where $r,s \in \Z$.

For $\bA \in \vv$ we have
\[\lambda_{r,s}(t) \cdot \bA = \begin{pmatrix} t^{r+s} \ba & t^{r-s}\bb \\ t^{-r+s}\bc & t^{-r-s}\bd \end{pmatrix}.\]
It follows that $\bA$ is not stable for $T$ if and only if at least one of $\ba, \bb, \bc, \bd$ is zero. However it's clear that the sets of matrices with each of these entries zero lie in the same $G$-orbit, so $\bA$ is not stable for $G$ if and only if it is in the same $G$-orbit as an $n$-matrix $\bA'$ with  $\bc' = 0$ as required.
\end{proof}

Let $\ww$ denote the set of upper-triangular $n$-matrices in $\vv$. Then the set $\uu$ of non-stable points is $G \cdot \ww$. In the notation of Equation \eqref{firstdecomp} we have
\begin{equation}\label{sgvu} \mathcal{S}_{G,\vv,\uu} = \mathcal{S}_{G,\vv,G \cdot \ww} = (G \cdot \ww)^2 \cap \mathcal{S}_{G,\vv} = G^2 \cdot \mathcal{S}_{G,\vv,\ww}\end{equation}
where $\mathcal{S}_{G,\vv,\ww} = \ww^2 \cap \mathcal{S}_{G,\vv}$, the set of pairs of upper triangular $n$-matrices not separated by invariants, and the action of $G^2$ is separately on the two arguments. We have the following somewhat surprising description of $\mathcal{S}_{G,\vv,\ww}$:

\begin{Lemma}\label{isom} There exists an isomorphism of algebraic varieties
\[\Phi:\mathcal{S}_{G,\vv,\ww} \cong \mathcal{N}_{G,\vv} \times \C^{2n}.\]
\end{Lemma} 

\begin{proof} Let $(\bA, \bA') \in \ww^2.$ Define an $n$-matrix $\bB = (B_1, \ldots, B_n) \in \vv$ as follows:
\[B_i = \begin{pmatrix} -a_i & a'_i \\ -d'_i & d_i \end{pmatrix}.\]
Then set $\Phi(\bA,\bA') = (\bB,\bb,\bb')$.
We claim that $\bB \in \mathcal{N}_{G,\vv}$ if and only if $(\bA,\bA') \in \mathcal{S}_{G,\vv}$. This proves the result, since we may freely choose $\bb$ and $\bb'$ without affecting $\bB$.

Now it follows from \cite[Corollary~3.3]{DerksenMakamDegreeBounds} that the set of invariants in $\C[\vv]^G$ with degree at most two form a zero-separating set. Therefore, we have $\bB \in \mathcal{N}_{G,\vv}$ if and only if
\[\det(B_i) = 0 \ \text{for all} \ 1 \leq i \leq n\]
and
\[\langle B_i|B_j \rangle = 0 \ \text{for all} \ 1 \leq i<j \leq n.\]
We now see that
\[\det(B_i) = -a_id_i + a'_id'_i = 0 \Leftrightarrow \det(A_i) = \det(A'_i),\]
and
\[\langle B_i | B_j \rangle= a_ia_j - a'_id'_j - d'_ia'_j +d_id_j - (d_i-a_i)(d_j-a_j) = -a'_id'_j -d'_ia'_j+a_id_j+d_ia_j = 0 \]
\[\Leftrightarrow -a'_id'_j -d'_ia'_j = -a_id_j-d_ia_j  \]
\[\Leftrightarrow a_ia_j+d_id_j - (a_i+d_i)(a_j+d_j) =  a'_ia'_j+d'_id'_j - (a'_i+d'_i)(a'_j+d'_j)   \]
\[\Leftrightarrow \langle A_i | A_j \rangle = \langle A'_i| A'_j \rangle.\]

This shows that $\bB \in \mathcal{N}_{G,\vv}$ whenever $(\bA,\bA') \in \mathcal{S}_{G,\vv,\ww}$, and that if $\bB \in \mathcal{N}_{G,\vv}$ we have $f(\bA)=f(\bA')$ for all $f \in \C[\vv]^G$ with degree $\leq 2$. To complete the proof it remains only to establish that $\xi(A_i,A_j,A_k,A_l) = \xi(A'_i,A'_j,A'_k,A'_l)$ whenever $\bB \in \mathcal{N}_{G,\vv}$. Now since $\bA,\bA' \in \ww$ and 
$\xi(A_i,A_j,A_k,A_l)$ is the coefficient of $e_ie_je_ke_l$ in the determinant
\[\begin{vmatrix} e_iA_i & e_jA_j\\ e_kA_k & e_lA_l \end{vmatrix} \]
\[= \begin{vmatrix} e_i a_i &e_i b_i & e_j a_j & e_j b_j\\
0 & e_i d_i & 0 & e_j d_j \\ e_k a_k & e_k b_k& e_l a_l & e_l b_l \\
0 & e_k d_k & 0 & e_l d_l \end{vmatrix} \]
\[= \begin{vmatrix} e_i a_i & e_j a_j &e_i b_i & e_j b_j\\
 e_k a_k & e_l a_l & e_k b_k& e_l b_l \\
0 &0 & e_i d_i  & e_j d_j \\
0 & 0 &e_k d_k & e_l d_l \end{vmatrix} \]
\[=(e_ie_la_ia_l-e_je_ka_ja_k)(e_ie_ld_id_l-e_je_kd_jd_k)\]
we have
\[\xi(A_i,A_j,A_k,A_l) = -a_ia_ld_jd_k-a_ja_kd_id_l,\]
and therefore
\[\xi(A_i,A_j,A_k,A_l) = \xi(A'_i,A'_j,A'_k,A'_l)\]
\[ \Leftrightarrow a_ia_ld_jd_k+a_ja_kd_id_l - a'_ia'_ld'_jd'_k-a'_ja'_kd'_id'_l = 0.\]4

Now we note that 
\begin{eqnarray*} a_ia_ld_jd_k+a_ja_kd_id_l - a'_ia'_ld'_jd'_k-a'_ja'_kd'_id'_l  &\\
= (a_kd_l+a_ld_k)(a_id_j+a_jd_i-a'_id'_j-a'_jd'_i) &\\
+(a_jd_l+a_ld_j)(a_id_k+a_kd_i-a'_id'_k-a'_kd'_i) &\\
-(a_id_l+a_ld_i)(a_jd_k+a_kd_j-a'_jd'_k-a'_kd'_j) &\\
-(a_jd_k+a_kd_j)(a_id_l+a_ld_i-a'_id'_l-a'_ld'_i) &\\
+(a_id_k+a_kd_i)(a_jd_l+a_ld_j-a'_jd'_l-a'_ld'_j) &\\
+(a_id_j+a_jd_i)(a_kd_l+a_ld_k-a'_kd'_l-a'_ld'_k).
\end{eqnarray*}

\begin{eqnarray*} 
= \langle A_k|A_l \rangle( \langle A_i | A_j \rangle - \langle A'_i | A'_j \rangle ) &\\
+ \langle A_j|A_l \rangle( \langle A_i | A_k \rangle - \langle A'_i | A'_k \rangle ) &\\
- \langle A_i|A_l \rangle( \langle A_j | A_k \rangle - \langle A'_j | A'_k \rangle ) &\\
- \langle A_j|A_k \rangle( \langle A_i | A_l \rangle - \langle A'_i | A'_l \rangle ) &\\
+ \langle A_i|A_k \rangle( \langle A_j | A_l \rangle - \langle A'_j | A'_l \rangle ) &\\
+ \langle A_i|A_j \rangle( \langle A_k | A_l \rangle - \langle A'_k | A'_l \rangle ) &\\
=0 &
\end{eqnarray*}
since $ \langle A_i | A_j \rangle = \langle A'_i | A'_j \rangle$ for all $1 \leq i<j \leq n$.  
\end{proof}

The decomposition of $\mathcal{N}_{G,\vv}$, where $G = \SL_l(\C) \times \SL_m(\C)$ acts in the usual way on $\vv = \M_{l,m}^n$ was studied by Burgin and Draisma \cite{BurginDraismaNullcone}. In the special case $l=m=2$, Theorem 1 in loc. cit. gives the following:

\begin{prop} The nullcone $\mathcal{N}_{G,\vv}$ decomposes as
\[\mathcal{N}_{G,\vv} = \mathcal{D}_r \cup \mathcal{D}_c\]
where
\[\mathcal{D}_r = \{\begin{pmatrix} \ba & \bb \\ \lambda \ba & \lambda \bb \end{pmatrix} \in \mathcal{N}_{G,\vv}: \ba, \bb \in \C^n, \lambda \in \C\}\]
and
\[\mathcal{D}_c = \{\begin{pmatrix} \ba & \lambda \ba \\  \bc & \lambda \bc \end{pmatrix} \in \mathcal{N}_{G,\vv}: \ba, \bc, \in \C^n, \lambda \in \C\}.\]

The sets $\mathcal{D}_r$ and $\mathcal{D}_c$ are closed and irreducible.
\end{prop}

Now using equation \eqref{sgvu} and Lemma \ref{isom} we obtain a decomposition of $\mathcal{S}_{G,\vv}$:

\begin{equation}\label{sgvdecomp}
\mathcal{S}_{G,\vv} = \overline{\Gamma_{G,\vv}} \cup G^2 \cdot \cc_r \cup G^2 \cdot \cc_c
\end{equation}
where $$\cc_r:= \Phi^{-1}( \mathcal{D}_r) = \{ \left( \begin{pmatrix} \ba & \bb \\ 0 &  \lambda \bd' \end{pmatrix},  \begin{pmatrix} \lambda \ba & \bb' \\ 0 &  \bd' \end{pmatrix}\right)\in \mathcal{S}_{G,\vv}: \ba, \bb, \bb', \bd' \in \C^n, \lambda \in \C\}$$
and
$$\cc_c:= \Phi^{-1}( \mathcal{D}_c) = \{\left(\begin{pmatrix} \ba & \bb \\ 0 &  \lambda \ba' \end{pmatrix},  \begin{pmatrix} \ba' & \bb' \\ 0 &  \lambda \ba \end{pmatrix} \right) \in \mathcal{S}_{G,\vv}: \ba, \ba' \bb, \bb' \in \C^n, \lambda \in \C\}.$$

We claim that $G^2 \cdot \cc_c$ and $G^2 \cdot \cc_r$ are closed (if so, it's clear they are also irreducible, being the orbit of a connected group on a vector space). We will need the following lemma, which seems to be well-known and was shown to me by Harm Derksen. For lack of a good reference we provide a proof:

\begin{Lemma}\label{harm}
Let $G$ be a linear algebraic group over $\kk$ and $B \leq G$ a Borel subgroup. Let $V$ be a vector space over $\kk$ on which $G$ acts linearly and let $W$ be a subspace of $V$. Suppose $B \cdot W \subseteq W$. Then $G \cdot W$ is closed.
\end{Lemma}

\begin{proof} Consider the set $Z:= \{(gB,v): g^{-1}v \in W\} \subseteq G/B \times V$. Let $\pi$ be the projection $G/B \times V \rightarrow V$. Then $G \cdot W = \pi(Z)$. As $G/B$ is a projective, hence complete, variety, the map $\pi$ is projective and takes closed sets to closed sets. Clearly $Z$ is closed, so $G \cdot W$ is closed as required.
\end{proof}

%clearly??

Now an easy direct calculation shows that $\cc_r$ and $\cc_c$ are stable under the action of $B^2$. Since this is a Borel subgroup of the linear algebraic group $G^2$, the claim is proved.

The dimension of both $\cc_c$ and $\cc_r$ is $4n+1$. Since $\dim(B^2) = 8$ and $B^2$ is in fact the stabliser of both, we get that the dimension of $G^2 \cdot \cc_r$ is $$\dim(G^2) + \dim(\cc_r) - \dim(B^2) = 12 + (4n+1) - 8 = 4n+5$$ and similarly for $G^2 \cdot \cc_c$. There are now two possibilities. The first is that $G^2 \cdot \cc_r$ and $G^2 \cdot \cc_c$ are contained inside $\overline{\Gamma_{G,\vv}}$, and hence $\mathcal{S}_{G,\vv} = \overline{\Gamma_{G,\vv}}$ is irreducible. In that case, our methods tell us nothing about the minimum size of a separating set. The second is that one or both of these components is not contained in $\overline{\Gamma_{G,\vv}}$. In that case, $\mathcal{S}_{G,\vv}$ has at least two components of different dimensions, one of which has dimension smaller than $4n+6 = 2 \dim(V) - \dim(\kvg)$. In that case, the smaller component has codimension strictly greater than $\kvg$ and, by Proposition \ref{krullbound}, no polynomial separating set exists.

Note that for $n \leq 3$ we must have $\mathcal{S}_{G,\vv} = \overline{\Gamma_{G,\vv}}$. This is because for $n \leq 3$, $\kvg$ is a polynomial ring, so it certainly contains a polynomial separating set, ruling out the second option.

So the final step in our argument is to describe $G^2 \cdot \cc_r \cap \overline{\Gamma_{G,\vv}}$ and $G^2 \cdot \cc_c \cap \overline{\Gamma_{G,\vv}}$. Note that since $G^2$ preserves $\overline{\Gamma_{G,\vv}}$ this is the same as describing $G^2 \cdot (\cc_r \cap \overline{\Gamma_{G,\vv}})$ and $G^2 \cdot (\cc_c \cap \overline{\Gamma_{G,\vv}})$. 

Now we claim:

\begin{Lemma}\label{rank} Let $(\bA,\bA') \in \mathcal{S}_{G,\vv,\ww}$. Define a $6 \times n$ matrix as follows:
\[m_{\bA,\bA'} = \begin{pmatrix} \ba \\ \bb \\ \bd \\ \ba' \\ \bb' \\ \bd' \end{pmatrix}.\]Then $(\bA,\bA') \in \overline{\Gamma_{G,\vv}}$ if and only if $\rk(m_{\bA,\bA'}) \leq 3$.
\end{Lemma}

\begin{proof} $\Leftarrow$: Recall that the diagonal commuting action of  $\GL_n(\C)$ on $\vv$ preserves $\overline{\Gamma_{G,\vv}}$, so it is enough to show that there exists $h \in \GL_n(\C)$ such that $(h \star \bA, h \star \bA') \in \overline{\Gamma_{G,\vv}}$.
Suppose $m_{\bA,\bA'}$ has rank $\leq 3$ and $(\bA, \bA') \not \in \overline{\Gamma_{G,\vv}}$.
Write $\bB = h \star \bA = (B_1,B_2, \ldots, B_n)$, and  $\bB' = h \star \bA' = (B'_1,B'_2, \ldots, B'_n)$.
 Now if $\rk(m_{\bA,\bA'})$, then there exists $h \in \GL_n(\C)$ such that $B_i = B'_i = 0$ for $i>3$. So by restricting to the first three arguments we obtain a pair 
$$(\bB,\bB') \in \mathcal{S}_{G,\M_{2,2}^3} \setminus \overline{\Gamma_{G,\M_{2,2}^3}}.$$ By the discussion preceding this Lemma, that is impossible. 

$\Rightarrow:$ This is trivial if $n \leq 3$, so assume $n>3$. Suppose $(\bA,\bA') \in \overline{\Gamma_{G,\vv}}$. This means that there exist functions $\bA(t),\bA'(t): \C^* \rightarrow \vv$ and $g(t), g'(t): \C^* \rightarrow \SL_2(\C)$ such that 
\begin{itemize}
\item $\bA'(t) = (g(t),g'(t))\cdot \bA(t)$ for all $t \in \C^*$;
\item $\lim_{t \rightarrow 0}\bA(t) = \bA, \lim_{t \rightarrow 0} \bA'(t) = \bA'$.
\end{itemize}

Write
\[g(t) = \begin{pmatrix} w(t) & x(t) \\ y(t) & z(t) \end{pmatrix}, g'(t) = \begin{pmatrix} w'(t) & x'(t) \\ y'(t) & z'(t) \end{pmatrix}\]
and
\[\bA(t) = \begin{pmatrix} \ba(t) & \bb(t) \\ \bc(t) & \bd(t) \end{pmatrix}, \bA'(t) = \begin{pmatrix} \ba'(t) & \bb'(t) \\ \bc'(t) & \bd'(t) \end{pmatrix}\] where throughout we abuse notation by using the same letter for a function and its limit. 

Then
\[m_{\bA,\bA'} = \lim_{t \rightarrow 0} m_{\bA(t),\bA'(t)}\]
and
\[m_{\bA(t),\bA'(t)} = \begin{pmatrix} \ba(t) \\ \bb(t) \\ \bd(t) \\ w(t)w'(t) \ba(t) + w(t)y'(t) \bb(t) + x(t)y'(t) \bd(t) \\ 
w(t)x'(t) \ba(t) + w(t)z'(t) \bb(t) + x(t) z'(t) \bd(t) \\ y(t)x'(t) \ba(t) + y(t) z'(t) \bb(t) + z(t) z'(t) \bd(t) \end{pmatrix}.\]
It is clear that $\rk(m_{\bA(t), \bA'(t)}) \leq 3$ for all $t \in \C^*$, since the last three rows are linear combinations of the first three. Since the set of $6 \times n$ matrices over $\C$ with rank bounded above  by some fixed number is closed, it follows that $\rk(m_{\bA,\bA'}) \leq 3$. 
\end{proof}

\begin{rem} For $n>3$, one can easily give examples of elements of $\cc_r$ or $\cc_c$ such that $\rk(m_{\bA,\bA'}) = 4$. Such elements will not lie in $\overline{\Gamma_{G,\vv}}$. Thus we have show that for $n>3$, $\mathcal{S}_{G,\vv}$ contains two components with dimension strictly less than $2- \dim(\kvg)$, i.e. codimension $> \kvg$. This shows that for $n>3$, no polynomial separating set for $\kvg$ exists. 
\end{rem}

To every element of $(\bA,\bA') \in \cc_r$ we associate a $4 \times n$ matrix 
\[m^r_{\bA,\bA'} = \begin{pmatrix} \ba \\ \bb \\ \bb' \\ \bd' \end{pmatrix}.\] Similarly to every element of $(\bA,\bA') \in \cc_c$ we associate a $4 \times n$ matrix 
\[m^c_{\bA,\bA'} = \begin{pmatrix} \ba \\ \bb \\ \bb' \\ \ba' \end{pmatrix}.\] 

Now from Lemma \ref{rank} we obtain:

\begin{cor}\
\begin{enumerate}
\item[(a)] $(\bA,\bA') \in \overline{\Gamma_{G,\vv}} \cap \cc_r$ if and only if $\rk(m^r_{\bA,\bA'}) \leq 3$;

\item[(b)] $(\bA,\bA') \in \overline{\Gamma_{G,\vv}} \cap \cc_c$ if and only if $\rk(m^c_{\bA,\bA'}) \leq 3$.
\end{enumerate}
\end{cor}

Consequently the dimension of $\overline{\Gamma_{G,\vv}} \cap \cc_r$ is $[3(n-3)+12]+1 = 3n+4$: we choose 4 vectors  $\ba,\bb,\bb',\bd'$ spanning a subspace of $\C^n$ with dimension 3, and are left with a free choice of $\lambda$ to fix $\ba'$ and $\bd$. The dimension of $\overline{\Gamma_{G,\vv}} \cap \cc_c$ is the same for similar reasons. Since the stabiliser of either is $B^2$ we have
\[\dim(\overline{\Gamma_{G,\vv}} \cap (G^2 \cdot \cc_r)) = \dim(\overline{\Gamma_{G,\vv}} \cap (G^2 \cdot \cc_r)) = 3n+8.\]

In addition we note that
\[\cc_r \cap \cc_c = \{\left(\begin{pmatrix} \ba & \bb \\ 0 &  \mu \lambda \ba \end{pmatrix},  \begin{pmatrix} \lambda \ba & \bb' \\ 0 &  \mu \ba \end{pmatrix} \right): \ba, \bb, \bb' \in \C^n, \lambda, \mu \in \C\}\] with dimension $3n+2$. It follows that
\[\dim(G^2 \cdot \cc_r \cap G^2 \cdot \cc_c) = 3n+6.\]
Furthermore, $\rk(m_{\bA,\bA'}) \leq 3$ for all $(\bA,\bA') \in \cc_r \cap \cc_c$, so that by Lemma \ref{rank} we get $\cc_c \cap \cc_r \subseteq \overline{\Gamma_{G,\vv}}$. Consequently we have $(G^2 \cdot \cc_c) \cap (G^2 \cdot \cc_r) \subseteq \overline{\Gamma_{G,\vv}}$ also. This completes the proof of Theorem \ref{sepvarstructure}.

The structure of the separating variety is described in the following ``Venn'' diagram:

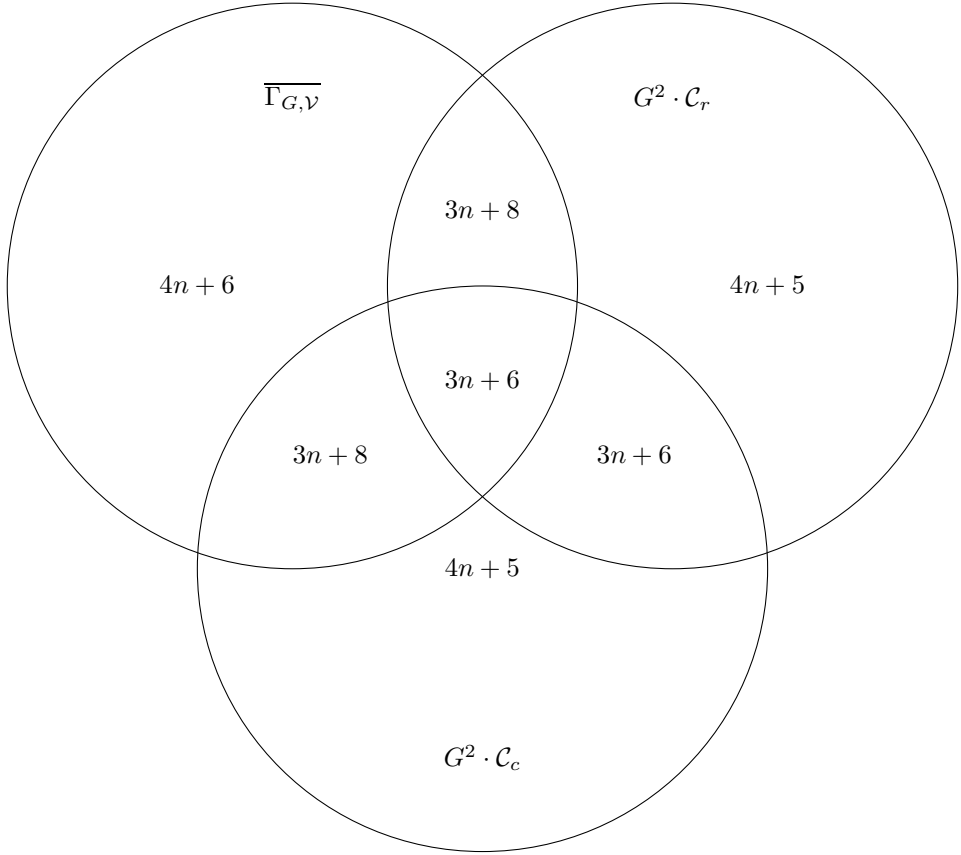
\begin{center}
\begin{tikzpicture}[scale=2.5]

\draw (2,2) circle[radius=1.5]; 
\draw (4,2) circle[radius=1.5];
\draw (3,0.5) circle[radius = 1.5];
\node at (2,3){ $\overline{\Gamma_{G,\vv}}$};
\node at (1.5,2){$4n+6$};
\node at (4,3){$G^2 \cdot \cc_r$};
\node at (4.5,2){$4n+5$};
\node at (3,0.5){$4n+5$};
\node at (3,-0.5){$G^2 \cdot \cc_c$};
\node at (3,2.4){$3n+8$};
\node at (2.2,1.1){$3n+8$};
\node at (3.8, 1.1){$3n+6$};
\node at (3,1.5){$3n+6$};
\end{tikzpicture}
\end{center}

In the diagram above, the number in a given region represents the dimension of the smallest intersection of the components containing that region. The analogous construction to the usual set-theoretic Venn diagram would have the numbers in a region representing the dimension of that region, but that does not make sense because not all regions in the diagram represent varieties.

One can see from the diagram that for $n>3$, $\mathcal{S}_{G,\vv}$ is not connected in codimension $n-3$. Applying Proposition \ref{grothbound} we find that the minimum possible size of a separating set for $\kvg$ is then
\[\dim(\kvg)+n-3 = 4n-6+n-3 = 5n-9.\]
This completes the proof of Theorem \ref{main22}. The table below compares this lower bound with the size of the separating set $S_n$ given in Proposition \ref{fft2}:
\begin{center}
\begin{table}[h]
\begin{tabular}{c|cccc}
$n$ & $\dim(\C[\M_{2,2}^n]^G)$ & $|S_n| $& Lower bound \\ \hline
%1 & 2 & 2 & \\
2 & 3 & 3 &3 \\
3 & 6 & 6&6\\
4 & 10 & 11& 11\\
5 & 14 & 20& 16\\
6 & 18 & 36 & 21 
\end{tabular}
\end{table}
\end{center}

In particular, our results show that no polynomial separating set for $\C[\M_{2,2}^n]^G$ exists when $n \geq 4$. Further, for $n \leq 4$ the given generating set is a minimum separating set, but for $n \geq 5$ it may not be.

%this section may or may not make the final cut
\section{Quiver Interpretation}
A {\it quiver} is a quadruple $Q = (Q_0,Q_1,t,h)$, consisting of two ordered sets $Q_0$ (vertices)  and $Q_1$ (arrows), along with two functions $t, h: Q_1 \rightarrow Q_0$ (tail and head respectively). It is usually visualised as a directed graph with a node for each element of $Q_0$, and for each $a \in Q_1$ a directed edge leading from $t(a)$ to $h(a)$.

Let $\kk$ be any field.  A {\it representation} $V$ of the quiver $Q$ over $\kk$ with dimension vector $\alpha$ is an assignment to each vertex $x \in Q_0$ of a vector space $V(x)$, and to each arrow $a \in Q_1$ of a linear map $V(a): V(t(a)) \rightarrow V(h(a))$, where we write $\alpha = (\alpha(x): x \in Q_0)$ and $\alpha(x) = \dim(V(x))$ for all $x \in Q_0$.

Let $V$ be a representation of the quiver $Q$ with $|Q_0| = k$, $|Q_1|=n$. By choosing a basis of each vector space $V(x)$, we may identify $V$ with the $n$-tuple of matrices $$\bA = (A_1, A_2, \ldots, A_n)$$ where $A_j \in \M_{\alpha(t(a_j)),\alpha(h(a_j))}$ for all $j$ is the matrix representing $V(a_j)$ with respect to the chosen basis, and a pair of representations is said to equivalent if the matrices associated to them can be made the same by choosing diffferent bases for the $V(x)$. Choosing a different basis is tantamount to replacing $\bA$ with 
$$g \cdot \bA:= (g_{t(a_1)} A_1 g^{-1}_{h(a_1)},  g_{t(a_2)} A_2 g^{-1}_{h(a_2)}, \ldots,g_{t(a_n)} A_n g^{-1}_{h(a_n)} )$$ where
$$g = (g_{x_1},g_{x_2}, \ldots, g_{x_k}) \in  \GL_{\alpha}(\kk):= \Pi_{i=1}^k \GL_{\alpha(x_i)}(\kk)$$ is the $k$-tuple of change of basis matrices where $g_{x_i}$ describes the change of basis on $V(x_i)$. Thus, we have an action of $\GL_{\alpha}(\kk)$ on 
\[\vv_{\alpha}:= \Pi_{i=1}^n \M_{\alpha(t(a_i)),\alpha(h(a_i))}\] and a pair of $n$-tuples of matrices $\bA, \bA' \in \vv_{\alpha}$ represent equivalent representations of $Q$ if and only if they lie in the same $\GL_{\alpha}(\kk)$-orbit.

There is a natural notion of direct sum for representations of a given quiver: if $V$ and $W$ are representations of $Q$ over $\kk$ with dimension vectors $\alpha$ and $\beta$ respectively, then $V \oplus W$ is the representation of $Q$ over $\kk$ with dimension vector $\alpha+\beta \in \Z^k$ defined by
\[(V \oplus W)(x) = V(x) \oplus W(x)\] for all $x \in Q_0$ and
\[(V \oplus W)(a) = V(a) \oplus W(a)\] for all $a \in Q_1$. A representation is said to be {\it indecomposable} if it cannot be written as a direct sum of two non-trivial representations. $Q$ is said to have:

\begin{enumerate}
\item {\it finite representation type} if $Q$ has only finitely many inequivalent indecomposable representations;
\item {\it tame representation type} if the inequivalent indecomposable representations of $Q$ in each dimension vector occur in finitely many one-parameter families;
\item {\it wild representation type} otherwise.
\end{enumerate}

We recommend \cite{DerksenWeyman} as a good source for learning more about the representation theory of quivers and its connection with invariant theory. 

Now let $G$ denote the subgroup $\SL_{\alpha}(\C):= \Pi_{i=1}^k \SL_{\alpha(x_i)}(\C)$ of $\GL_{\alpha}(\C)$. The algebra $\C[\vv_{\alpha}]^G$ is called the {\it algebra of semi-invariants} associated to $Q$ with dimension vector $\alpha$. A remarkable result linking representation type of quivers and invariant theory is the following (see \cite{SkowronskiWeyman} for proof):

\begin{prop}[Skowronski-Weyman]\label{SW}  Let $Q$ be a quiver. Then the following are equivalent:
\begin{enumerate}
\item $Q$ has tame representation type;

\item $\C[\vv_{\alpha}]^G$ is a polynomial ring or hypersurface for each dimension vector $\alpha$.
\end{enumerate}
\end{prop}

Along similar lines, Sato and Kimura \cite{SatoKimura} showed that if $Q$ has finite representation type then $\C[\vv_{\alpha}]^G$ is polynomial for all dimension vectors $\alpha$. We make the following conjecture generalising Proposition \ref{SW}:

\begin{conj}  Let $Q$ be a quiver. Then the following are equivalent:
\begin{enumerate}
\item $Q$ has tame representation type;

\item $\C[\vv_{\alpha}]^G$ contains a polynomial or hypersurface separating set, for each dimension vector $\alpha$.
\end{enumerate}
\end{conj}

Here a hypersurface separating set is a separating set generating a hypersurface, i.e. with cardinality $\dim(\C[\vv_{\alpha}]^G+1)$.
Since generating sets are separating sets, the forward direction is known. To prove the conjecture, it remains to show that if $Q$ has wild representation type, then there exists a dimension vector $\alpha$ for which $|S| > \dim(\C[\vv_{\alpha}]^G+1)$ for every separating set $S \subseteq \C[\vv_{\alpha}]^G$.

\begin{eg} Representations of the quiver $Q_n$ with diagram

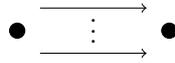
\begin{center}
\begin{tikzpicture}
\draw[fill] (-1,0) circle[radius=0.1];
%\node at (-1,0){$l$};
\draw[fill] (1,0) circle[radius=0.1];
%\node at (1,0){$m$};
\draw[->] (-0.7,0.3) -- (0.7,0.3);
\node at (0,0.1){$\vdots$};
\draw[->] (-0.7,-0.3) -- (0.7,-0.3);
\end{tikzpicture}
\end{center}

with $n$ arrows and dimension vector $(l,m)$ over $\C$ can be identified with $n$-tuples of $l \times m$ matrices, i.e. elements of $\M_{l,m}^n(\C)$, and two such are isomorphic if they lie in the same $\GL_{l}(\C) \times \GL_m(\C)$ orbit, where the action is by simultaneous left- and right- multiplication. The quiver is of finite type for $n=1$, tame for $n=2$, and wild for $n \geq 3$. Let $G = \SL_l(\C) \times \SL_m(\C)$. Theorem \ref{main22} shows that the algebra of invariants $\C[\M_{2,2}^n]^G$ does not contain a hypersurface separating algebra if $n\geq 5$, while $\C[\M_{2,2}^n]^G$ is known to be a polynomial algebra for $n \leq 3$, and a hypersurface for $n=4$. To verify the conjecture for this quiver it remains to show that, for some dimension vector $\alpha = (l,m)$ and $n=3$ or $n=4$, $\C[\M_{l,m}^n]^G$ does not contain a polynomial or hypersurface separating set. 
\end{eg}

\begin{eg} Representations of the quiver $Q_n$

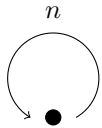
\begin{center}
\begin{tikzpicture}
\draw[fill] (0,0) circle[radius=0.1];
%\node at (-1,0){$l$};

%\node at (1,0){$m$};
\draw[->] (0.3,0) arc (-60:240:0.6);
\node at (0,1.4) {$n$};
\end{tikzpicture}
\end{center}
where the $n$ represents $n$ separate arrows and $\alpha = l$ can be identified with $n$-tuples of $l \times l$ matrices, i.e. elements of $\M_{l,l}^n$, with a pair of such representations equivalent if and only if those $n$-matrices are simultaneously conjugate under the action of $\GL_{l}(\C)$. Let $G = \SL_l(\C)$. The semi-invariant rings of such actions were studied in \cite{Elmer2x2}. The quiver $Q_n$ is of tame type if $n=1$ and wild otherwise. Theorem 1.4 in loc. cit. shows that $\C[\M_{2,2}^n]^G$ does not contain a hypersurface separating set for $n \geq 4$, and $\C[\vv_{2,2}^n]^G$ is known to be polynomial if $n=2$ and a hypersurface if $n=3$. Thus to verify the conjecture for this quiver  it remains to show that, for some dimension $l$ and $n=2$ or $n=3$, $\C[\M_{l,l}^n]^G$ does not contain a polynomial or hypersurface separating set.  
\end{eg}

\bibliographystyle{plain}
\bibliography{MyBib}

\begin{thebibliography}{10}

\bibitem{Bedratyuk7}
Leonid Bedratyuk.
\newblock A complete minimal system of covariants for the binary form of degree
  7.
\newblock {\em J. Symbolic Comput.}, 44(2):211--220, 2009.

\bibitem{BurginDraismaNullcone}
Matthias B\"{u}rgin and Jan Draisma.
\newblock The {H}ilbert null-cone on tuples of matrices and bilinear forms.
\newblock {\em Math. Z.}, 254(4):785--809, 2006.

\bibitem{DerksenKemper}
Harm Derksen and Gregor Kemper.
\newblock {\em Computational invariant theory}.
\newblock Invariant Theory and Algebraic Transformation Groups, I.
  Springer-Verlag, Berlin, 2002.
\newblock Encyclopaedia of Mathematical Sciences, 130.

\bibitem{DerksenMakamDegreeBounds}
Harm Derksen and Visu Makam.
\newblock Polynomial degree bounds for matrix semi-invariants.
\newblock {\em Adv. Math.}, 310:44--63, 2017.

\bibitem{DerksenWeyman}
Harm Derksen and Jerzy Weyman.
\newblock {\em An introduction to quiver representations}, volume 184 of {\em
  Graduate Studies in Mathematics}.
\newblock American Mathematical Society, Providence, RI, 2017.

\bibitem{Dolgachev}
Igor Dolgachev.
\newblock {\em Lectures on invariant theory}, volume 296 of {\em London
  Mathematical Society Lecture Note Series}.
\newblock Cambridge University Press, Cambridge, 2003.

\bibitem{DomokosSemi}
M.~Domokos.
\newblock Characteristic free description of semi-invariants of {$2\times2$}
  matrices.
\newblock {\em J. Pure Appl. Algebra}, 224(5):106220, 13, 2020.

\bibitem{DomokosPoincare}
M\'{a}ty\'{a}s Domokos.
\newblock Poincar\'{e} series of semi-invariants of {$2\times 2$} matrices.
\newblock {\em Linear Algebra Appl.}, 310(1-3):183--194, 2000.

\bibitem{DufresnePhD}
Emilie Dufresne.
\newblock Separating invariants.
\newblock {\em Ph. D. thesis, Kingston Ontario}, 2008.

\bibitem{DufresneSeparating}
Emilie Dufresne.
\newblock Separating invariants and finite reflection groups.
\newblock {\em Adv. Math.}, 221:1979--1989, 2009.

\bibitem{DufresneElmerSezer}
Emilie Dufresne, Jonathan Elmer, and M\"{u}fit Sezer.
\newblock Separating invariants for arbitrary linear actions of the additive
  group.
\newblock {\em Manuscripta Math.}, 143(1-2):207--219, 2014.

\bibitem{DufresneKohlsFiniteSep}
Emilie Dufresne and Martin Kohls.
\newblock A finite separating set for {D}aigle and {F}reudenburg's
  counterexample to {H}ilbert's fourteenth problem.
\newblock {\em Comm. Algebra}, 38(11):3987--3992, 2010.

\bibitem{DufresneKohlsSepVar}
Emilie Dufresne and Martin Kohls.
\newblock The separating variety for the basic representations of the additive
  group.
\newblock {\em J. Algebra}, 377:269--280, 2013.

\bibitem{Eisenbud}
David Eisenbud.
\newblock {\em Commutative algebra with a view toward algebraic geometry},
  volume 150 of {\em Graduate Texts in Mathematics}.
\newblock Springer-Verlag, New York, 1995.

\bibitem{Elmer2x2}
Jonathan Elmer.
\newblock The separating variety for $2\times 2$ matrix invariants.
\newblock {\em Linear and Multilinear Algebra, to appear}, 2022.

\bibitem{ElmerKohls}
Jonathan Elmer and Martin Kohls.
\newblock Separating invariants for the basic {$\Bbb G_{a}$}-actions.
\newblock {\em Proc. Amer. Math. Soc.}, 140(1):135--146, 2012.

\bibitem{FleischmannNoetherBound}
Peter Fleischmann.
\newblock The {N}oether bound in invariant theory of finite groups.
\newblock {\em Adv. Math.}, 156(1):23--32, 2000.

\bibitem{FogartyNoetherBound}
John Fogarty.
\newblock On {N}oether's bound for polynomial invariants of a finite group.
\newblock {\em Electron. Res. Announc. Amer. Math. Soc.}, 7:5--7 (electronic),
  2001.

\bibitem{Giral}
Jos\'{e}~M. Giral.
\newblock Krull dimension, transcendence degree and subalgebras of finitely
  generated algebras.
\newblock {\em Arch. Math. (Basel)}, 36(4):305--312, 1981.

\bibitem{KamkeKemper}
Tobias Kamke and Gregor Kemper.
\newblock Algorithmic invariant theory of nonreductive groups.
\newblock {\em Qual. Theory Dyn. Syst.}, 11(1):79--110, 2012.

\bibitem{KemperCompRed}
Gregor Kemper.
\newblock Computing invariants of reductive groups in positive characteristic.
\newblock {\em Transform. Groups}, 8(2):159--176, 2003.

\bibitem{KemperSeparating}
Gregor Kemper.
\newblock Separating invariants.
\newblock {\em Journal of Symbolic Computation}, 44(9):1212 -- 1222, 2009.
\newblock Effective Methods in Algebraic Geometry.

\bibitem{Lopatin3x3}
A.~A. Lopatin.
\newblock Relatively free algebras with the identity {$x^3=0$}.
\newblock {\em Comm. Algebra}, 33(10):3583--3605, 2005.

\bibitem{Mumford}
D.~Mumford, J.~Fogarty, and F.~Kirwan.
\newblock {\em Geometric invariant theory}, volume~34 of {\em Ergebnisse der
  Mathematik und ihrer Grenzgebiete (2)}.
\newblock Springer-Verlag, Berlin, third edition, 1994.

\bibitem{SatoKimura}
M.~Sato and T.~Kimura.
\newblock A classification of irreducible prehomogeneous vector spaces and
  their relative invariants.
\newblock {\em Nagoya Math. J.}, 65:1--155, 1977.

\bibitem{SezerCyclic}
M\"{u}fit Sezer.
\newblock Explicit separating invariants for cyclic {$P$}-groups.
\newblock {\em J. Combin. Theory Ser. A}, 118(2):681--689, 2011.

\bibitem{SkowronskiWeyman}
A.~Skowro\'{n}ski and J.~Weyman.
\newblock The algebras of semi-invariants of quivers.
\newblock {\em Transform. Groups}, 5(4):361--402, 2000.

\bibitem{WehlauCyclicViaClassical}
David~L. Wehlau.
\newblock Invariants for the modular cyclic group of prime order via classical
  invariant theory.
\newblock {\em J. Eur. Math. Soc. (JEMS)}, 15(3):775--803, 2013.

\end{thebibliography}

\end{document}